\documentclass[12pt]{amsart}
\usepackage{amssymb,amsmath,amsthm}
\usepackage{enumerate}

\DeclareMathOperator{\sign}{sign}
\DeclareMathOperator{\Ran}{Ran}
\DeclareMathOperator{\Dom}{Dom}
\DeclareMathOperator{\Ker}{Ker}
\DeclareMathOperator{\Tr}{Tr}

\DeclareMathOperator{\spec}{spec}
\DeclareMathOperator*{\slim}{s-lim}

\renewcommand\Im{\hbox{{\rm Im}}\,}
\renewcommand\Re{\hbox{{\rm Re}}\,}
\newcommand{\abs}[1]{\lvert#1\rvert}
\newcommand{\Abs}[1]{\left\lvert#1\right\rvert}
\newcommand{\norm}[1]{\lVert#1\rVert}
\newcommand{\Norm}[1]{\left\lVert#1\right\rVert}

\newcommand{\R}{{\mathbb R}}
\newcommand{\C}{{\mathbb C}}

\newcommand{\calH}{{\mathcal H}}
\newcommand{\calK}{{\mathcal K}}
\newcommand{\calF}{\mathcal{F}}
\newcommand{\calN}{\mathcal{N}}
\newcommand{\calC}{\mathcal{C}}

\numberwithin{equation}{section}

\theoremstyle{plain}
\newtheorem{theorem}{\bf Theorem}[section]
\newtheorem{lemma}[theorem]{\bf Lemma}
\newtheorem{proposition}[theorem]{\bf Proposition}
\newtheorem{assumption}[theorem]{\bf Assumption}

\theoremstyle{definition}
\newtheorem{definition}[theorem]{\bf Definition}

\theoremstyle{remark}
\newtheorem*{remark*}{\bf Remark}

\newcommand{\wt}{\widetilde}

\begin{document}

\title[Discontinuous functions of self-adjoint operators]{Spectral theory of 
discontinuous functions of self-adjoint operators  and scattering theory}

\author{Alexander Pushnitski}
\address{Department of Mathematics,
King's College London, 
Strand, London, WC2R~2LS, U.K.}
\email{alexander.pushnitski@kcl.ac.uk}

\author{Dmitri Yafaev}
\address{Department of Mathematics, University of Rennes-1, 
Campus Beaulieu, 35042, Rennes, France}
\email{yafaev@univ-rennes1.fr}

\dedicatory{To the memory of M.~Sh.~Birman (1928--2009)}

\begin{abstract}
In the smooth scattering theory framework, we consider a pair of self-adjoint operators
$H_0$, $H$ and discuss the spectral projections of these operators corresponding 
to the interval $(-\infty,\lambda)$.
The purpose of the paper is to study the  spectral properties of the 
difference $D(\lambda)$ of these spectral projections. 
We completely describe the absolutely continuous 
spectrum of the operator $D(\lambda)$
in terms of the eigenvalues of the scattering matrix $S(\lambda)$  
for the  operators $H_{0}$ and $H$. We also prove that the singular continuous 
spectrum of the operator $D(\lambda)$ is empty.
\end{abstract}

\subjclass[2000]{Primary 47A40; Secondary 47B25}

\keywords{Scattering matrix, Carleman operator, absolutely continuous spectrum, spectral projections}

\maketitle

%%%%%%%%%%%%%%%%%%%%%%%%%%%%%%%%%%%%%%%%%%%%%%%%%
%%%%%%%%%%%%%%%%%%%%%%%%%%%%%%%%%%%%%%%%%%%%%%%%%
\section{Introduction}\label{sec.a}
%%%%%%%%%%%%%%%%%%%%%%%%%%%%%%%%%%%%%%%%%%%%%%%%%
%%%%%%%%%%%%%%%%%%%%%%%%%%%%%%%%%%%%%%%%%%%%%%%%%

%%%%%%%%%%%%%%%%%%%%%%%%%%%
 
%%%%%%%%%%%%%%%%%%%%%%%%%%%
Let $H_0$ and $H$ be self-adjoint operators in a Hilbert space $\calH$  and suppose that the difference 
$V=H-H_0$ is a compact operator.  
If $\varphi:\R\to\R$ is a continuous function which tends to zero at infinity then a well known 
simple argument shows that the difference
\begin{equation}
\varphi(H)-\varphi(H_0)
\label{a1}
\end{equation}
is compact. On the other hand, if $\varphi$ has discontinuities on the 
essential spectrum of $H_0$ and  $H$, then the difference
\eqref{a1} may fail to be compact; see \cite{Krein,KM}.

The simplest example of a function $\varphi$ with a discontinuity
is the characteristic function of a semi-axis. Thus, for a Borel  set $\Lambda\subset\R$ 
we denote by $E_0(\Lambda)$ (resp. $E(\Lambda)$) the spectral projection
of $H_0$ (resp. $H$) corresponding to the set $\Lambda$
and consider the difference
\begin{equation}
D(\lambda)=E((-\infty,\lambda))-E_0((-\infty,\lambda))
\label{a2}
\end{equation}
where $\lambda$ belongs to the absolutely continuous (a.c.) spectrum of $H_0$.

In  \cite{Krein}, M.~G.~Kre\u{\i}n has shown that
under some assumptions of the trace class type on the pair $H_0$ and $H$,
the operator $\varphi(H)-\varphi(H_0)$ belongs to the trace class 
for all sufficiently ``nice'' functions $\varphi$ and 
$$
\Tr(\varphi(H)-\varphi(H_0))=\int_{-\infty}^\infty \varphi'(t) \xi(t) dt,
$$
where the function $\xi(\cdot)=\xi(\cdot; H,H_{0})$ is known as the spectral shift function. 
Formally taking   the characteristic function $\chi_{(-\infty,\lambda)}$ of the interval 
$(-\infty,\lambda)$ for $\varphi$, we obtain the relation
\begin{equation}
\xi(\lambda)=-\text{``$\Tr$''} D(\lambda)
\label{a10}
\end{equation}
where ``$\Tr$'' is the regularized trace.

The relation between the spectral shift function and 
the   scattering matrix $S(\lambda)=S(\lambda;H,H_0)$
for the pair $H_0$, $H$ was found in the paper 
\cite{BK} by M.~Sh.~Birman and M.~G.~Kre\u{\i}n, 
where it was shown that 
\begin{equation}
\det S(\lambda)=e^{-2\pi i\xi(\lambda)}
\label{a11}
\end{equation}
for a.e. $\lambda$ from the core   of the  a.c.  spectrum of $H_0$
(see e.g.  \cite[Section 1.3]{Yafaev} for the discussion of the notion of the core). 
The importance  of \eqref{a10}, \eqref{a11} is in the fact that they give a relation between  
the key object of  \emph{spectral perturbation theory} $D(\lambda)$ 
and  the key object of \emph{scattering theory} $S(\lambda)$.

Our aim here is to discuss the spectral properties of 
$D(\lambda)$. It turns out that  \eqref{a10}, \eqref{a11}
is not the only link between 
$D(\lambda)$ and $S(\lambda)$. 
In fact, 
the spectral properties of 
$D(\lambda)$ can be completely described in terms of 
the eigenvalues $ e^{i\theta_n(\lambda)}$, $n=1,\ldots, N$, 
$N\leq \infty$, distinct from $1$  of the  scattering matrix $S(\lambda)$. We show that the 
a.c. spectrum of $D(\lambda)$ consists of the union of the intervals
\begin{equation}
\bigcup_{n=1}^{N}[-\varkappa_n(\lambda),\varkappa_n(\lambda)],
\qquad
\varkappa_n(\lambda)=\abs{e^{i\theta_n(\lambda)}-1}/2,
\label{a0}
\end{equation}
where each interval has the multiplicity one in the 
spectrum. We also prove that the singular continuous spectrum 
of $D(\lambda)$ is empty, the eigenvalues of $D(\lambda)$ can accumulate
only to $0$ and to the points $\pm\varkappa_n(\lambda)$, and 
all eigenvalues of $D(\lambda)$ distinct from $0$ and
$\pm\varkappa_n(\lambda)$ have finite multiplicity.
In particular, $D(\lambda)$ is compact if and only if $S(\lambda)=I$.
On the other hand,   the a.c. spectrum 
of $D(\lambda)$ covers the interval $[-1,1]$   if and only if 
the spectrum of $S(\lambda)$ contains $-1$. 

The present paper can be considered as a continuation of 
\cite{Push}, where the description \eqref{a0} of the 
a.c. spectrum of $D(\lambda)$ was obtained 
using a combination of assumptions of trace class and smooth scattering theory. 
In contrast to 
\cite{Push}, here we use only the technique of smooth scattering 
theory, which yields stronger results.

Our ``model" operator is constructed in terms of a certain Hankel  
integral operator with kernel \eqref{b0a} which we call the 
``half-Carleman'' operator and of the scattering matrix. 
Using the  explicit diagonalization of the half-Carleman operator, 
given by the Mehler-Fock transform (see Section~\ref{sec.c1}),  
we find a class of operators smooth with respect to the ``half-Carleman'' operator. 
This allows us to develop scattering theory for the pair 
consisting of the model operator and the operator $D(\lambda)^2$.
To a certain extent, we were inspired by J.~S.~Howland's 
papers \cite{Howland} where the smooth version of scattering theory was 
developed via the Mourre commutator method.

There is a close relationship between the properties of the difference 
$\varphi(H)-\varphi(H_0)$ and the theory of Hankel operators. 
This fact was exhibited in the work \cite{Peller1} by V.~Peller. 
The problem discussed in this paper gives another example of this relationship.

When this paper was at the final stage of preparation, 
the authors have learnt that their teacher M.~Sh.~Birman has passed away. Much of the modern spectral and scattering theory is Birman's legacy. We dedicate this work to his memory.

%\pagebreak

\section{Main results} 

%%%%%%%%%%%%%%%%%%%%%%%%%%%
\subsection{Definition of the operator  $H$}\label{sec.a2}
%%%%%%%%%%%%%%%%%%%%%%%%
Let $H_0$ be a self-adjoint operator in a Hilbert space $\calH$, and let $V$ be 
a symmetric operator which we consider as a perturbation of $H_0$. 
Our first goal is to correctly define the sum $H=H_0+V$.
Following the approach which goes back at least to \cite{Kato2} and is
developed in more detail in  \cite[Sections 1.9, 1.10]{Yafaev}, 
below we define the operator $H$ in terms of its resolvent.
If $V$ is bounded, then the   operator $H$ we define coincides with the 
operator sum $H_0+V$. 
In the semi-bounded case the   operator $H$ can be defined via its quadratic form.

 We suppose that $V$ is factorized as $V=G^*JG$, 
where $G$ is an operator from $\calH$ to an auxiliary
Hilbert space $\calK$ and $J$ is an operator in $\calK$.  
We assume that 
\begin{equation}
\begin{split}
&J=J^* \text{ is bounded in $\calK$, } 
\\
\Dom \abs{H_0}^{1/2} \subset& \Dom G \quad 
\text{ and } \quad
G(\abs{H_0}+I)^{-1/2}\text { is compact.}
\end{split}
\label{a3}
\end{equation}
In applications such a  factorization
often arises naturally from the structure of the problem.
In any case, one can always take
$\calK=\calH$, $G=\abs{V}^{1/2}$ and $J=\sign(V)$. 

Let us accept

\begin{definition}\label{def1}
A self-adjoint operator $H$ corresponds to the sum $H_{0}+V$ if the following two conditions are satisfied: 

\begin{enumerate}[\rm (i)]

\item
For any regular point $z\in\C\setminus \spec(H )$, its resolvent $R(z) =(H-zI)^{-1}$ admits the representation
\begin{equation}
R(z)= (\abs{H_0}+I)^{-1/2}B (z)(\abs{H_0}+I)^{-1/2}
\label{a9} 
\end{equation}
where the operator $B (z)$ is bounded. 
In particular, $\Dom H\subset \Dom \abs{H_{0}}^{1/2}$. 

\item
One has 
$$
(f_0,Hf)=(H_0f_0,f)+(JGf_0,Gf), 
\quad \forall f_0\in\Dom H_0,
\quad \forall f\in\Dom H.
$$
\end{enumerate}
\end{definition}

Only one self-adjoint operator $H$ can satisfy this definition, 
and under the assumption \eqref{a3} such an operator exists 
and is defined below via its resolvent. 
For $z\in\C\setminus \spec(H_0)$, let us denote
$R_0(z)=(H_0-zI)^{-1}$. Formally, we  define the operator $T(z)$
(sandwiched resolvent) by  
\begin{equation}
T(z)=GR_0(z)G^*;
\label{a6a}
\end{equation}
more precisely, this means
$$
T(z)=(G(\abs{H_0}+I)^{-1/2}) (\abs{H_0}+I)R_0(z)(G(\abs{H_0}+I)^{-1/2})^*.
$$
By \eqref{a3}, the operator $T(z)$ is  compact.
Under the assumption \eqref{a3}, it can be shown (see \cite[Sections 1.9,1.10]{Yafaev})
that the operator $I+T(z)J$ has a bounded inverse for all $z\in\C\setminus\R$ and
\begin{equation}
R(z)=R_0(z)-(G R_0(\overline z))^* J (I+T(z)J)^{-1} GR_0(z)
\label{a7}
\end{equation}
is the resolvent of a self-adjoint operator $H$ which satisfies   Definition~\ref{def1}.
Of course the resolvents of $H_{0}$ and $H$ are related by the usual identity
\begin{equation}
R(z)-R_0(z)=-(GR_0(\overline z))^*JGR(z).
\label{res}
\end{equation}

If $H_0$ is semi-bounded 
from below, then \eqref{a3} means that $V$ is 
$H_0$-form compact, and then 
$H$ coincides with the operator 
$H_0+V$ defined as a quadratic form sum
(see the KLMN Theorem in  \cite{RS2}).

%%%%%%%%%%%%%%%%%%%%%%%%%%%
\subsection{Scattering Theory}\label{sec.a3}
%%%%%%%%%%%%%%%%%%%%%%%%%%%

Recall that, for a pair of self-adjoint operators $H_{0}$ and $H$ and a Borel set 
${\Lambda}\subset{\mathbb R}$, the (local) wave operators are introduced by the relation
\[
W_{\pm} (H,H_{0};\Lambda)=\slim_{t\to \pm \infty}e^{iHt}e^{-iH_{0}t}E_{0}(\Lambda) P_{0}^{(a)}
\]
provided these strong limits exist. 
Here and in what follows we denote by  $P_{0}^{(a)}$ (resp. $P^{(a)}$) the orthogonal projection
onto the absolutely continuous subspace   of $H_{0}$ (resp. $H$). 
The wave operators enjoy the intertwining property 
$W_{\pm} (H,H_{0};\Lambda) H_{0}= H W_{\pm} (H,H_{0};\Lambda)$. 
The wave operators are called complete if 
\[
\Ran W_{+} (H,H_{0};\Lambda)
=
\Ran W_{-} (H,H_{0};\Lambda)
=
\Ran \big(E (\Lambda) P^{(a)}\big).
\]
If $\Lambda= {\mathbb R}$, then $\Lambda$ is omitted from the notation.

We fix a compact interval $\Delta\subset\R$ and assume that 
the spectrum of $H_0$ in $\Delta$ is purely a.c.
with a constant multiplicity $N_{0}\leq\infty$. 
The interior of $\Delta$ is denoted by $ {\rm int}(\Delta)$. We make an assumption typical for smooth scattering theory; 
in the terminology of \cite{Yafaev}, we
assume that $G$ is strongly $H_0$-smooth 
on $\Delta$ with some exponent $\alpha\in (0, 1]$.
This means the following. 
Let $\calF$ be a unitary operator from $\Ran E_0(\Delta)$ to 
$L^2(\Delta,\calN)$, $\dim\calN=N_{0}$, such that $\calF$
diagonalizes $H_0$: if $f\in\Ran E_0(\Delta)$ then 
\begin{equation}
(\calF  H_0 f)(\lambda)
=
\lambda (\calF f)(\lambda), \quad \lambda\in\Delta.
\label{a8}
\end{equation}
The strong $H_0$-smoothness of  $G$ on the interval $\Delta$ means that the operator 
\[
G_\Delta\overset{\rm def}{=}GE_0(\Delta) : \Ran E_0(\Delta) \to \calK
\] 
satisfies the equation 
\begin{equation}
(\calF G_\Delta^*\psi)(\lambda)
=
Z(\lambda)\psi, 
\quad 
\forall \psi\in\calK,
\quad
\lambda\in\Delta,
\label{a3b}
\end{equation}
where $Z=Z(\lambda):\calK\to\calN$
is a family of compact operators obeying
\begin{equation}
\norm{Z(\lambda)}\leq C,
\quad
\norm{Z(\lambda)-Z(\lambda')}\leq C\abs{\lambda-\lambda'}^\alpha,
\quad
\lambda, \lambda'\in\Delta.
\label{a3a}
\end{equation}
Note that the notion of strong smoothness is not unitary invariant,
as it depends on the choice of the map $\calF$. It follows from \eqref{a3b} that the adjoint operator $  G_\Delta \calF^* : L^2(\Delta,\calN)\to \calK$ acts by the formula
\begin{equation}
G_\Delta \calF^* f
=
\int_{\Delta}Z(\lambda)^* f(\lambda) d\lambda.
\label{a3B}\end{equation}

Let us summarize our assumptions:

\begin{assumption}\label{as1}
\begin{enumerate}[\rm (A)]
\item
$H_0$ has a purely a.c. spectrum with the multiplicity $N_0$ on $\Delta$.
\item
$V$ admits a factorization $V=G^*JG$ with the operators $G$ and $J$ satisfying \eqref{a3}.
\item
$G$ is strongly $H_0$-smooth on $\Delta$. % see \eqref{a8}--\eqref{a3a}.
\end{enumerate}
\end{assumption}

We need the following well known results (see e.g. \cite[Section~4.4]{Yafaev}).

\begin{proposition}\label{pr1}
Let Assumption~\ref{as1} hold. 
Then the operator-valued function $T(z)$   
defined by \eqref{a6a}
is   
H\"older continuous for $\Re z\in {\rm int}(\Delta)$, $\Im z\geq 0$. 
The set $\mathcal{ X} \subset\Delta$ where the equation 
$f+T(\lambda+ i0)Jf=0$ has a nontrivial solution is closed 
and has the Lebesgue measure zero.
The operator $I+T(\lambda+i0)J$ is invertible for all 
$\lambda\in\Omega\overset{\rm def}{=}  {\rm int}(\Delta)\setminus \mathcal{ X}$.
\end{proposition}

\begin{proposition}\label{pr2}
Let Assumption~\ref{as1} hold. Then the local wave operators
$W_\pm(H,\\ H_0;\Delta)$ exist and are complete. Moreover, the spectrum of $H$ in $\Omega$ 
is purely absolutely continuous.
\end{proposition}

The last statement of  Proposition~\ref{pr2} is usually formulated 
under the additional assumption $\Ker G=\{0\}$. 
Actually, this assumption is not necessary; this is verified    in Lemma~A.1 of  the Appendix.

In terms of the wave operators the (local) scattering operator is defined as
\[
\mathbf{S}=\mathbf{S}(H,H_0;\Delta)= W_+(H,H_0;\Delta)^* W_-(H,H_0;\Delta).
\]
The scattering operator $\mathbf{S}$
commutes with $H_{0}$ and is unitary on the subspace $\Ran E_0(\Delta)$. 
Thus, we have a representation
\[
( \calF \mathbf{S}  \calF^* f)(\lambda) =S(\lambda) f(\lambda), 
\quad \text{ a.e. }\lambda\in \Delta,
\]
where the operator  $S(\lambda):\calN\to\calN$  is  called the scattering matrix
for the pair of operators $H_0$, $H$. 
The scattering matrix  is a unitary operator in $\calN$.

We need the stationary representation for the scattering matrix
(see \cite[Chapter~7]{Yafaev} for the details).

%%%%%%%%%%%%%%%%
\begin{proposition}\label{pr3}
%%%%%%%%%%%%%%%%
Let Assumption~\ref{as1} hold, and let $\lambda\in\Omega$. Then
\begin{equation}
S(\lambda)=I-2\pi i Z(\lambda)J(I+T(\lambda+i0)J)^{-1}Z(\lambda)^*.
\label{a6}\end{equation}
\end{proposition}

This proposition, in particular, implies that $S(\lambda)$ is 
a H\"older continuous function of $\lambda\in\Omega$. 

Since the operator $Y(\lambda)=J(I+T(\lambda+i0)J)^{-1}$  
is bounded and $Z(\lambda)$ is compact, it follows that the 
operator $S(\lambda)-I$ is  compact.  
Thus, the spectrum of $S(\lambda)$ consists of  eigenvalues 
accumulating possibly only to the point 1. 
All eigenvalues of $S(\lambda)$ distinct from $1$ have finite multiplicity.
If $N_{0}=\infty$ then necessarily  $1$ is the eigenvalue of infinite multiplicity 
or the accumulation point (or both).

%%%%%%%%%%%%%%%%%%%%%%%%%%%
\subsection{Main Result}\label{sec.a4}
%%%%%%%%%%%%%%%%%%%%%%%%%%%

%Here is our main result.
First note that since $D(\lambda)$ is the difference
of two orthogonal projections, the spectrum of 
$D(\lambda)$ is a subset of $[-1,1]$. 

We denote by $e^{i\theta_n(\lambda)}$, $n=1,\dots,N$, 
the eigenvalues of $S(\lambda)$ distinct from $1$. 
The eigenvalues are enumerated with the multiplicities taken 
into account.  
We set
$\varkappa_n(\lambda)=\abs{e^{i\theta_n(\lambda)}-1}/2$.

\begin{theorem}\label{th1}
Let Assumption~\ref{as1} hold true and let $\lambda\in\Omega$. 
Then   
the a.c. spectrum of $D(\lambda)$
consists of the union of intervals \eqref{a0},
where each interval has the multiplicity one in the 
spectrum. The operator 
$D(\lambda)$
has no singular continuous spectrum.
The eigenvalues of $D(\lambda)$ can accumulate
only to $0$ and to the points $\pm\varkappa_n(\lambda)$. 
All eigenvalues of $D(\lambda)$ distinct from $0$ and
$\pm\varkappa_n(\lambda)$ have finite multiplicities.
\end{theorem}

The part of the theorem concerning the a.c.
spectrum can be equivalently stated as follows:
\emph{
The a.c. component of $D(\lambda)$ 
is unitarily equivalent to the operator of multiplication
by $x$ in the orthogonal sum}
$$
\bigoplus_{n=1}^{N} L^2([-\varkappa_n(\lambda),\varkappa_n(\lambda)],dx).
$$
% Of course, if some of $\varkappa_n(\lambda)$ are zero, then  the corresponding 
% terms can be dropped from the direct sum over $n$. 

In \cite{Push}, the above characterization of the a.c. spectrum 
of $D(\lambda)$ was obtained under 
more restrictive assumptions 
which combined smooth type and trace class type requirements. 
The construction of 
\cite{Push} gives no information on either the singular spectrum of 
$D(\lambda)$ or on its eigenvalues.

%%%%%%%%%%%%%%%%%%%%%%%%%%%
\subsection{Examples}
%%%%%%%%%%%%%%%%%%%%%%%%%%%
Let $H_0=-\Delta$ in $L^2(\R^d)$ with $d\geq1$.
Application of the Fourier transform shows that 
$H_0$ has a purely a.c. spectrum $[0,\infty)$ 
with multiplicity $N=2$ if $d=1$ and $N=\infty$ if $d\geq2$.

Let $H=H_0+V$, where $V$ 
is the operator of multiplication by a function $V:\R^d\to\R$
which is assumed to satisfy
\begin{equation}
\abs{V(x)}\leq C(1+\abs{x})^{-\rho},
\quad 
\rho>1.
\label{a4}
\end{equation}
Let $G=\abs{V}^{1/2}$, $J=\sign{V}$ so that $V=G^*JG$. 
Then   Assumption~\ref{as1} is fulfilled on every 
compact subinterval $\Delta$ of $ (0,\infty)$.
 Moreover,  by a well known argument involving 
Agmon's ``bootstrap" \cite{Agmon}
and Kato's theorem \cite{Kato} on the absence of positive eigenvalues
of $H$, the operator $I+T(\lambda+i0)J$ is invertible
for all $\lambda>0$  and hence $\Omega=(0,\infty)$.  Thus,
  Proposition~\ref{pr2} implies that the wave operators $W_{\pm} (H,H_{0})$   exist and are complete (this result was first obtained in  \cite{Kato3,Kuroda}).
The scattering matrix $S(\lambda)$  is a unitary operator in $L^2(\mathbb S^{d-1})$
(here  $\mathbb S^0=\{-1,1\}$) and depends H\"older continuously on $\lambda>0$. 
According to Proposition~\ref{pr3} the operator $S(\lambda)-I$
  is compact, and hence its spectrum consists of eigenvalues $  e^{i\theta_n(\lambda)} $. 
  
  In this example all the assumptions of Theorem~\ref{th1}
hold true with $\Omega=(0,\infty)$. Denoting, as before, $\varkappa_n(\lambda)=\abs{e^{i\theta_n(\lambda)}-1}/2$,
we obtain:
%%%%%%%%%%%%%%%%%%%%%%
\begin{theorem}\label{th2}
%%%%%%%%%%%%%%%%%%%%%%
Assume \eqref{a4}. Then for any $\lambda>0$, the a.c. spectrum of $D(\lambda)$
consists of the union of intervals   \eqref{a0}, 
where each interval has a multiplicity one in the 
spectrum. The operator 
$D(\lambda)$
has no singular continuous spectrum.  
The eigenvalues of $D(\lambda)$ can accumulate
only to $0$ and to the points $\pm\varkappa_n(\lambda)$. 
All eigenvalues of $D(\lambda)$ distinct from $0$ and
$\pm\varkappa_n(\lambda)$ have finite multiplicities.
\end{theorem}
The above characterisation of the a.c. spectrum
was   obtained earlier in \cite{Push} for $d=1,2,3$ 
under the more restrictive assumption
$\rho>d$.

Similar applications are possible in situations where the diagonalization 
of $H_0$ is known explicitly. For example, 
the perturbed
Schr\"odinger operator with a constant magnetic field in dimension 
three 
(and probably the perturbed periodic 
Schr\"odinger operator in arbitrary dimension) 
can be considered. 
Moreover, in Theorem~\ref{th1}, we do not assume the operators $H_0$, $H$ to be
semibounded. 
Thus, one can apply this theorem to the perturbations of the Dirac operator
and the Stark operator (i.e. the Schr\"odinger operator with a linear electric potential). 

%%%%%%%%%%%%%%%%%%%%%%%%%%%
\subsection{The strategy of the proof of Theorem~\ref{th1}}
%%%%%%%%%%%%%%%%%%%%%%%%%%%
In order to simplify our notation, we will assume without the loss of generality 
that $\Delta=[-1,1]$ and $\lambda=0\in\Omega$.
Clearly, the general case can be reduced to this one by a shift and scaling.
We fix $a>0$ such that $[ -a, a]\subset\Omega$. 
% Let us denote $\R_+=(0,\infty)$, $\R_-=(-\infty,0)$ and let
% $E_0(\R_\pm)$, $E(\R_\pm)$ be the corresponding spectral projections. 
In Section~\ref{sec.b} by using a simple operator theoretic argument
(borrowed from \cite{Push}),
we reduce the spectral analysis of $D(0)$ to 
the spectral analysis of the self-adjoint operators
\begin{equation}
M_+=E_0(\R_+)E(\R_-)E_0(\R_+),
\qquad
M_-=E_0(\R_-)E(\R_+)E_0(\R_-),
\label{a5}
\end{equation}
where as usual $\R_+=(0,\infty)$, $\R_-=(-\infty,0)$.
In Section~\ref{sec.c}, we construct an explicit ``model'' self-adjoint operator 
$M$ and analyze its spectrum. After this, in Sections~\ref{sec.b} and \ref{sec.d} we 
prove that the wave operators for the pair $W_{\pm}(M_+, M)$ exist and 
are complete. This allows us to describe the spectrum
of $M_+$. The operator $M_-$ is analyzed in a similar way.

The proof of the existence and completeness of the wave operators for 
the pair   $M$, $M_+$ is achieved by showing that the difference 
$M_+-M$ can be represented as $X K X$, where
the operator $X$ is strongly $M$-smooth and the operator $K$ is 
compact, see Section~\ref{sec.b2}. 
In \cite{Push} the same aim was achieved, roughly speaking, by showing that 
(under more stringent assumptions) the difference $M_+-M$
is a trace class operator.

%%%%%%%%%%%%%%%%%%%%%%%%%%%%%%%%%%%%%%%%%%%%%%%%%
%%%%%%%%%%%%%%%%%%%%%%%%%%%%%%%%%%%%%%%%%%%%%%%%%
\section{The model operator}\label{sec.c}
%%%%%%%%%%%%%%%%%%%%%%%%%%%%%%%%%%%%%%%%%%%%%%%%%
%%%%%%%%%%%%%%%%%%%%%%%%%%%%%%%%%%%%%%%%%%%%%%%%%

\subsection{The half-Carleman operator $\calC_a$}\label{sec.c1}
The Carleman operator is the Hankel integral operator in 
$L^2(\R_+)$ with the integral kernel $1/(x+y)$. Let $\calC_a$ be
the integral operator on $L^2(0,a)$ with the Carleman kernel
(up to a normalization $1/\pi$): 
\begin{equation}
\calC_a(x,y)=\frac1{\pi}\frac1{x+y}.
\label{b0a}
\end{equation}
We will call $\calC_a$ the half-Carleman operator. 

Our first task is to recall the explicit diagonalization formula 
for $\calC_a$. 
Essentially, this diagonalization 
is given by   Mehler's formula  (see e.g. \cite[formula (3.14.6)]{BE}):
\begin{equation}
\label{c1}
\frac1\pi\int_1^\infty\frac{P_{-\frac12+it}(y)}{x+y}dy
=
\frac{1}{\cosh(\pi t)}P_{-\frac12+it}(x),
\quad t\in\R.
\end{equation}
Here $P_\nu$ is the Legendre function.

Let us exhibit the unitary operator which 
diagonalizes $\calC_a$. 
Recall that the Mehler-Fock transform (see e.g. \cite[Section 3.4]{Yakub}
and references therein) is a unitary operator 
$U:L^2((1,\infty),dx)\to L^2((0,\infty),dt)$ defined for $g\in C_0^\infty(1,\infty)$ by 
\begin{equation}
(U g)(t)=\sqrt{t\tanh(\pi t)}\int_1^\infty P_{-\frac12+it}(x) g (x)dx, 
\quad
t\in(0,\infty).
\label{c3}
\end{equation}
Let us introduce the unitary operators $B_{1}: L^2((0,\infty),dt) \to L^2((0,1),d\mu)$ and
$B_{2}: L^2((0,a),du)\to L^2((1,\infty),dx)$ by the formulas
\[
(B_{1}h)(\mu)=\frac{\cosh(\pi t)}{\sqrt{\pi\sinh(\pi t)}} h(t), \quad \mu=\frac1{\cosh(\pi t)}\in(0,1),
\]
and
\[
(B_{2}f)(x)=\frac {\sqrt{a}}{x} f (a/x), \quad x\in(1,\infty).
\]
Then the operator $\mathcal U_a =B_{1}U B_{2}:L^2((0,a),du)\to L^2((0,1),d\mu)$
is also unitary.
Using the change of variables  $u=a/x$ in \eqref{c3}, we see that $\mathcal U_a$ acts as
\begin{equation}
\label{c5}
(\mathcal U_a f)(\mu)
=
\sqrt{\frac{a}{\pi} t \cosh(\pi t)}
\int_0^a P_{-\frac12+it}(a/u)\frac{f(u)}{u}du,
\quad
\mu=\frac1{\cosh(\pi t)}\in(0,1).
\end{equation}
 Changing the variables $u=a/x$, $v=a/y$ in 
Mehler's formula \eqref{c1}, we get
\begin{equation}
(\mathcal U_a \calC_a f)( \mu)= \mu (\mathcal U_a f)(\mu),
\qquad
\mu  \in(0,1).
\label{c11}
\end{equation}
Let us summarize the above calculations. 

%%%%%%%%%%%%%%%%%%%%%%%
\begin{lemma}\label{lma.c2}
The half-Carleman operator $\calC_a$ in $L^2(0,a)$ 
has  a purely a.c. spectrum of multiplicity one, 
$\spec(\calC_a)=[0,1]$. The explicit diagonalization \eqref{c11} of 
$\calC_a$ is given by the unitary operator $\mathcal U_a$ 
defined by \eqref{c5}. 
\end{lemma}

\subsection{The strong $\calC_a$--smoothness}\label{sec.c1a}
It turns out that  
the operators of multiplication by functions 
with a logarithmic decay at $x=0$ are strongly $\calC_a$-smooth. 
Before discussing this, 
we need   some bounds on the 
Legendre function:
%%%%%%%%%%%%%%%%%%%%%%%%
\begin{lemma}\label{lma2}
%%%%%%%%%%%%%%%%%%%%%%%%
For any $R>0$ there exist  constants $C_1(R)$, $C_2(R,\delta)$ such that
for any $x\geq 1$ and any $t,t_1,t_2\in[0,R]$, one has 
\begin{align}
\abs{P_{-\frac12-it}(x)}&\leq C_1(R) x^{-1/2}, 
\label{c6}
\\
\Abs{P_{-\frac12-it_2}(x)-P_{-\frac12-it_1}(x)}&\leq 
C_2(R,\delta)\abs{t_2-t_1}^\delta x^{-1/2}(1+\log x)^\delta.
\label{c7}
\end{align}
\end{lemma}
The proof is given in the Appendix.

Let the operator $X_\gamma^{(0)}$ act in the space   $L^2(0,a)$  by the formula
\begin{equation}
(X_\gamma^{(0)} f) (x)=(1+\abs{\log x})^{-\gamma} f (x), \quad x\in(0,a),
\quad \gamma>0.
\label{c12}
\end{equation}
 Similarly to \eqref{a3b}, we define the operator 
$\mathcal{Z}(\mu): L^2(0,a) \to {\mathbb C}$ for $\mu \in (0,1)$ by the equation
\begin{equation}
(\mathcal U_a X_\gamma^{(0)} f)(\mu)= \mathcal{Z}(\mu) f. 
% \quad   f\in C_0^\infty(0,a).
\label{uz}
\end{equation}

%%%%%%%%%%%%%%%%%%%%%%%%
\begin{lemma}\label{lma1}
%%%%%%%%%%%%%%%%%%%%%%%%
Let $\delta\in(0,1]$. Then for any $\gamma>\delta+1/2$,  
the operator $X_\gamma^{(0)}$ is strongly $\calC_a$-smooth 
with the exponent $\delta$ 
on any compact subinterval of $(0,1)$. 
\end{lemma}

\begin{proof}
In view of  \eqref{c5}, \eqref{c12} and \eqref{uz} the operator
$\mathcal{Z}(\mu)$      satisfies the equation
$$
\mathcal{Z}(\mu) f
=
\sqrt{\frac{a}{\pi} t \cosh(\pi t)}
\int_0^a P_{-\frac12+it}(a/u) (1+\abs{\log u})^{-\gamma}  \frac{f(u)}{u}du,
$$
where $\mu=1/{\cosh(\pi t)}$.
We have to check the estimates (cf. \eqref{a3a})
\begin{equation}
\norm{\mathcal{Z}(\mu)}\leq C,
\quad
\norm{\mathcal{Z}(\mu)- \mathcal{Z}(\mu')}\leq C\abs{\mu-\mu'}^\delta
\label{a3acz}
\end{equation}
on any compact subinterval of $(0,1)$. 
If  $\mu$ is bounded away from zero, then  the variable $t$ belongs to 
the interval $[0,R]$ with some $R<\infty$.
It follows  from Lemma~\ref{lma2}
that the function
$$
 P_{-\frac12+it}(a/u) (1+\abs{\log u})^{-\gamma}  (1/u)
$$
of $u\in (0,a)$
belongs to the space $L^2((0,a),du)$ for any $\gamma> 1/2$ and as an element of this space 
is H\"older continuous in $t\in[0,R]$ with the exponent $\delta<\gamma-1/2$. 
Since the map $\mu \mapsto t$ is continuously differentiable 
away from $\mu=1$, the required claim follows.
\end{proof}

\subsection{The operator $M $}
Here we define the operator $M $ which we consider 
as a ``model operator" for $M_+$ (recall that $M_+$ is defined by \eqref{a5}). 
Our goal will be  to prove 
that the wave operators $W_{\pm}(M_{+}, M)$
  exist and are complete.

First consider the operator $\calC_a^2$ in $L^2(0,a)$; obviously 
this operator has the integral kernel 
\begin{equation}
\calC_a^2(x,y)=\frac1{\pi^2}\int_0^a \frac{dt}{(x+t)(y+t)}.
\label{uzu}
\end{equation}
Lemmas~ \ref{lma.c2} and \ref{lma1} yield the following result.

%%%%%%%%%%%%%%%%%%%%%%%%
\begin{lemma}\label{lma1X}
%%%%%%%%%%%%%%%%%%%%%%%%
The operator
$\calC_a^2$ has a purely a.c. spectrum $[0,1]$ of 
multiplicity one and for any $\delta\in(0,1]$ and any $\gamma>\delta+1/2$
the operator $X^{(0)}_\gamma$ 
is strongly $\calC_a^2$-smooth with the  exponent $\delta$ on any compact subinterval of $(0,1)$.  
\end{lemma}

%recall that we assume $\lambda_0=0$, $\Delta=[-1,1]$. 

Next,   in $L^2((0,a),\calN)=L^2(0,a)\otimes \calN$  consider  the operators
\begin{equation}
M_1=\calC_a^2\otimes \Gamma,
\quad
X_\gamma^{(1)}=X_\gamma^{(0)}\otimes I,
\label{uzuz}
\end{equation}
where 
\begin{equation}
\Gamma= \frac14(S(0)-I)(S(0)^*-I)= \frac12 (I-\Re S(0)).
 \label{uzuz1}
\end{equation}
 At the last step we have used the unitarity of the scattering matrix.
The operator $\Gamma$ has a pure point spectrum with 
  the eigenvalues   $\varkappa_n(0)^2$.
From Lemma~\ref{lma1X}  it follows that, apart from the possible 
zero eigenvalue of infinite multiplicity,  $M_1$ has a purely a.c. spectrum
$\cup_{n=1}^N [0,\varkappa_n(0)^2]$ (each interval has a multiplicity one). 
Moreover, using the diagonalization of $\calC_a^2$ and 
choosing the basis of the eigenfunctions of $\Gamma$
in $\calN$, 
we can diagonalize  the operator $M_{1}$ in an obvious way. 
With respect to this diagonalization,  
for any $\delta\in(0,1]$ and any $\gamma>\delta+1/2$ the operator 
$X_\gamma^{(1)}$ is strongly $M_1$-smooth
with the exponent $\delta$ on any compact interval which   contains neither 
$0$ nor $\varkappa_n(0)^2 $, $n=1,\ldots, N$. 

Finally, we ``transplant" the operators $M_1$ and $X_\gamma^{(1)}$ into $\calH$. 
Recall (see Section~\ref{sec.a3}) that $\calF:\Ran E_0([-1,1])\to L^2( [-1,1],\calN)$ 
is a unitary operator which diagonalizes $H_0$. 
Let $\calH_a=\Ran E_0((0,a))$. 
It will be convenient to consider the 
restriction $\calF_a=\calF|_{\calH_a}$. 
Clearly, $\calF_a:\calH_a\to  L^2((0,a),\calN)$
is a unitary operator. 

Let us define the operators $M $, $X_\gamma$ in $\calH$ by 
\begin{equation}
M =\calF_a^* M_1 \calF_a\oplus 0,
\quad
X_\gamma=\calF_a^* X_\gamma^{(1)}\calF_a\oplus I
\label{c8}
\end{equation}
with respect to the orthogonal sum decomposition 
$\calH=\calH_a\oplus \calH_a^\bot$. Clearly, 
\begin{equation}
\begin{split}
 X_\gamma&=\omega_{\gamma}(H_{0}), \quad  \text{where}
\\
\omega_{\gamma}(x)&=(1+\abs{\log x})^{-\gamma}\chi_{(0,a)} (x)+\chi_{[a, \infty)} (x) +\chi_{(-\infty,0]} (x).
\end{split}
\label{omega}\end{equation}

%An important property of $X_\gamma$ which will be 
%exploited in the rest of the paper is that $X_\gamma$ commutes  with $H_0$. 

From the above analysis we obtain:
%%%%%%%%%%%%%%%%%%%%%%%%
\begin{theorem}\label{th.c2}
%%%%%%%%%%%%%%%%%%%%%%%%
Besides the eigenvalue at $0$ (possibly, of infinite multiplicity), 
the spectrum of $M$ is absolutely continuous. 
The a.c. spectrum of $M $ consists of the union 
$\cup_{n=1}^N [0,\varkappa_n(0)^2]$, where each interval has the  multiplicity one.
For any $\delta\in(0,1]$ and any $\gamma>\delta+1/2$ 
the operator $X_\gamma$ is strongly $M $-smooth
with the exponent $\delta$ on any compact interval which   contains neither 
$0$ nor $ \varkappa_n(0)^2$, $n=1,\ldots, N$. 
\end{theorem}

%%%%%%%%%%%%%%%%%%%%%%%%%%%%%%%%%%%%%%%%%%%%%%%%%
%%%%%%%%%%%%%%%%%%%%%%%%%%%%%%%%%%%%%%%%%%%%%%%%%
\section{Proof of Theorem~\ref{th1}}\label{sec.b}
%%%%%%%%%%%%%%%%%%%%%%%%%%%%%%%%%%%%%%%%%%%%%%%%%
%%%%%%%%%%%%%%%%%%%%%%%%%%%%%%%%%%%%%%%%%%%%%%%%%

%%%%%%%%%%%%%%%%%%%%%%%%%%%%%%%%%%%
\subsection{Reduction to the products of spectral projections}
%%%%%%%%%%%%%%%%%%%%%%%%%%%%%%%%%%%

Let us denote $D=D(0)$ and 
$$
\calH_+=\Ker(D-I), \quad 
\calH_-=\Ker(D+I),\quad
\calH_0=(\calH_-\oplus \calH_+)^\bot.
$$
It is well known (see e.g. \cite{ASS} or \cite{Halmos}) that 
$\calH_0$ is an invariant subspace for $D$ and that 
\begin{equation}
\text{$D|_{\calH_0}$
is unitarily equivalent to $(-D)|_{\calH_0}$.}
\label{b0}
\end{equation}
Thus, the spectral analysis of $D$ reduces to the spectral analysis
of $D^2$ and to the calculation of the dimensions of $\calH_+$
and $\calH_-$. 

 Recall that by our assumptions, 
the operator $I+ T(\lambda+i0) J$ is invertible for all $\abs{\lambda}\leq a$, 
  and $H$ has a purely a.c. spectrum on $[-a,a]$.
Using the notation $M_{+}$, $M_{-}$ (see \eqref{a5}) 
and the fact that $E(\{0\})=E_0(\{0\})=0$, 
by a simple algebra one obtains
\begin{equation}
D^2
=M_++M_-
=
E_0(\R_+)E(\R_-)E_0(\R_+)+ E_0(\R_-)E(\R_+)E_0(\R_-).
\label{b1}
\end{equation}
Clearly,  the r.h.s. provides a block-diagonal decomposition
of $D^2$ with respect to the orthogonal sum
$\calH=\Ran E_0(\R_-)\oplus \Ran E_0(\R_+)$.

Denote $\varkappa_n=\varkappa_n(0)$. 
Below we prove 

%%%%%%%%%%%%%%%%%%%%
\begin{theorem}\label{th.b1}
%%%%%%%%%%%%%%%%%%%%
Let Assumption~\ref{as1} hold true
and $\Delta=[-1,1]$, $\lambda=0$. Then the a.c. spectrum of $M_\pm$  
consists of the union of intervals $\cup_{n=1}^N[0,\varkappa_n^2]$,
where each interval has the multiplicity one 
in the spectrum.
The operators $M_+$ and $M_-$ have no singular continuous
spectrum.  
The eigenvalues of $M_\pm$ can accumulate only to $0$ 
and to the points $\varkappa_n^2$. 
All eigenvalues of $M_\pm$ distinct from $0$ and $\varkappa_n^2$ 
have finite multiplicities. 
\end{theorem}

From Theorem~\ref{th.b1} and the decomposition \eqref{b1} we immediately 
obtain that $D^2$ has no singular continuous spectrum; the a.c. spectrum 
of $D^2$ consists of the union of intervals $\cup_{n=1}^N[0,\varkappa_n^2]$,
where each interval has the multiplicity two; the eigenvalues of 
$D^2$ can accumulate only to $0$ 
and to the points $\varkappa_n^2$, and all eigenvalues of $D^2$ 
distinct from $0$ and $\varkappa_n^2$  have finite multiplicities.

From the above description of the spectrum of $D^2$ and from
\eqref{b0} we obtain the description of the spectrum of 
$D|_{\calH_0}$. 
In order to complete the proof of Theorem~\ref{th1}, 
it remains to consider the dimensions of $\calH_+$ and $\calH_-$. 
Assume first that $\varkappa_n<1$ for all $n$. Then $1$ cannot be an eigenvalue of $D^2$ of infinite multiplicity. Therefore
  $\dim\calH_-<\infty$ and $\dim\calH_+<\infty$
and we are done. 
If $\varkappa_n=1$ for some $n$, then the statement of Theorem~\ref{th1} is true
regardless of the dimensions of $\calH_+$ and $\calH_-$.
 
 Thus, for the proof of Theorem~\ref{th1} it suffices to
  prove Theorem~\ref{th.b1}. We consider the operator   $M_+$; 
the proof for $M_-$ is analogous.

%%%%%%%%%%%%%%%%%%%%%%%%%%%%%%%%%%%
\subsection{Application of scattering theory}\label{sec.b2}
%%%%%%%%%%%%%%%%%%%%%%%%%%%%%%%%%%%
Our proof of Theorem~\ref{th.b1} is based on the following well known 
fact from scattering theory, 
see e.g. \cite[Theorems 4.6.4, 4.7.9, 4.7.10]{Yafaev}.
%We state this fact  in a form convenient for our purposes; 
%the theorems in \cite{Yafaev} are more general.
%%%%%%%%%%%%%%%%%%%%%%%%
\begin{proposition}\label{prp.b2}
%%%%%%%%%%%%%%%%%%%%%%%%
Suppose that a bounded self-adjoint operator $M$ 
has a purely a.c. spectrum of constant multiplicity on an
open interval $\Lambda$. Suppose that a bounded operator $X$
%$\Ker X=\{0\}$
 is strongly $M$-smooth  with an 
exponent $\delta>0$ on every compact subinterval of $\Lambda$. 
Let $K$ be a compact self-adjoint operator 
and $\widetilde M=M+X^*KX$. 
Then the local wave operators $W_{\pm} (\widetilde M, M; \Lambda)$ 
for  $M$, $\widetilde M$ and the interval 
$\Lambda$ exist and are complete. 
Thus, the a.c. spectrum of $\widetilde M$ on $\Lambda $ has the same
multiplicity as that of $M$. 
Moreover, if $\delta >1/2$ then $\widetilde M$ has no singular continuous spectrum or 
eigenvalues of infinite multiplicity on $\Lambda$. 
The eigenvalues of $\widetilde M$ in $\Lambda$  can accumulate only to the endpoints of $\Lambda$. 
\end{proposition} 

In what follows we prove 

\begin{theorem}\label{th.b3}
Let Assumption~\ref{as1} hold true. Then 
for any $\gamma>0$, 
the difference $M_+ -M $ can be represented as 
$X_\gamma K X_\gamma$ where $K$ is a compact self-adjoint operator.
\end{theorem}

Given Theorem~\ref{th.b3}, we are in a position to prove  Theorem~\ref{th.b1} (for $M_+$). 
Let us assume that 
$\varkappa_{n}$ are enumerated such that 
$\varkappa_{n} \geq \varkappa_{n+1}$ for all $n$.  
Take any $n$ such that $\varkappa_{n} > \varkappa_{n+1}$ and
let us apply Proposition~\ref{prp.b2} to the pair $M$, $\widetilde M=M_{+}$ 
and the  interval   $\Lambda_{n}= (\varkappa_{n+1}^2,\varkappa_{n}^2)$. 
If $N<\infty$, then we also consider the  interval $\Lambda_{N}= (0,\varkappa_{N}^2)$. 
By Theorem~\ref{th.c2}, the operator $X_\gamma$ for $\gamma>1$ 
is strongly $M$-smooth with some  $\delta>1/2$
on all compact subintervals of $\Lambda_n$.  
Thus, it follows from Proposition~\ref{prp.b2} that 
the local wave operators $W_{\pm}(M_{+}, M ;  \Lambda_{n})$ 
for all $n$ exist and are complete. 
This implies (see   e.g. \cite[Theorem 4.6.5]{Yafaev}) that the global 
wave operators $W_{\pm}(M_{+}, M  )$ also exist and are complete. 
In particular, the a.c. parts of $M$ and $M_{+}$ are unitarily equivalent. 
Furthermore, since $\delta>1/2$ the conclusions of  Theorem~\ref{th.b1} 
about the singular spectrum of $M_{+}$ and its eigenvalues also
follow from Proposition~\ref{prp.b2}.
Thus, we have proven Theorem~\ref{th.b1} for $M_+$; 
the proof for $M_-$ is analogous. 

%%%%%%%%%%%%%%%%%%%%%%%%%%%%%%%%%%%
\subsection{The proof of Theorem~\ref{th.b3}}
%%%%%%%%%%%%%%%%%%%%%%%%%%%%%%%%%%%

The proof of Theorem~\ref{th.b3} consists of several steps 
which we proceed to outline. 
In this subsection, we state four lemmas; the proofs will be given
in Section~\ref{sec.d}.
The first two lemmas show that
only a neighborhood of the point $\lambda=0$ is essential
 for the analysis of the operator $M_{+}$.

Define
\begin{align}
M_2&=E_0(\R_+)E((-a,0))E_0(\R_+),
\label{M2}
\\
M_3&=E_0((0,a))E((-a,0))E_0((0,a)).
\label{M3}
\end{align}

\begin{lemma}\label{lma.b4}
For any $\gamma>0$, 
the difference $M_+ -M_2$ can be represented as 
$X_\gamma K X_\gamma$ where $K$ is a compact self-adjoint operator.
\end{lemma}

\begin{lemma}\label{lma.b5}
For any $\gamma>0$, 
the difference $M_2-M_3$ can be represented as 
$X_\gamma K X_\gamma$ where $K$ is a compact self-adjoint operator.
\end{lemma}

Below we use the fact 
that the operator $R_0(z)E_0(\R_+)$  
is analytic in $z\in\C\setminus[0,\infty)$ and so for any $\lambda<0$ 
the operator $R_0(\lambda)E_0(\R_+)$ is well defined, bounded and 
self-adjoint. 
Let $\mathcal D\subset\calH$ be the dense set
\begin{equation}
\mathcal D=\{f\in\calH\mid \exists \delta=\delta(f): \  E_0((-\delta,\delta))f=0\}.
\label{b1a}
\end{equation}
Recall our notation $Y(\lambda)=J(I+T(\lambda+i0)J)^{-1}$ 
(see Section~\ref{sec.a3}) and set $\Im Y(\lambda)=(Y(\lambda)-Y(\lambda)^*)/2i$.
Let us introduce an auxiliary  operator $M_4$ in terms of its quadratic form 
\begin{equation}
(M_4f,f)=
-\frac1\pi\int_{-a}^0 
((\Im Y(0))GR_0(\lambda)E_0((0,a))f,GR_0(\lambda)E_0((0,a))f)d\lambda
\label{b1b}
\end{equation}
for $f\in\mathcal D$.   

\begin{lemma}\label{lma.b6}
Formula \eqref{b1b} defines a bounded self-adjoint operator $M_4$ on $\calH$. 
%There exists  a bounded self-adjoint 
%operator $M_4$ on $\calH$ such that equality
% \eqref{b1b} is true. 
For any $\gamma>0$, 
the difference $M_3-M_4$ can be represented as 
$X_\gamma K X_\gamma$ where $K$ is a compact self-adjoint operator.
\end{lemma}

\begin{lemma}\label{lma.b7}
For any $\gamma>0$, 
the difference $M_4-M $ can be represented as 
$X_\gamma K X_\gamma$ where $K$ is a compact self-adjoint operator.
\end{lemma}

Clearly, Theorem~\ref{th.b3} and hence Theorem~\ref{th.b1} 
follow from Lemmas~\ref{lma.b4}, \ref{lma.b5}, \ref{lma.b6}, and \ref{lma.b7}.

%%%%%%%%%%%%%%%%%%%%%%%%%%%%%%%%%%%%%%%%%%%%%%%%%
%%%%%%%%%%%%%%%%%%%%%%%%%%%%%%%%%%%%%%%%%%%%%%%%%
\section{Proofs of Lemmas~\ref{lma.b4}--\ref{lma.b7}}\label{sec.d}
%%%%%%%%%%%%%%%%%%%%%%%%%%%%%%%%%%%%%%%%%%%%%%%%%
%%%%%%%%%%%%%%%%%%%%%%%%%%%%%%%%%%%%%%%%%%%%%%%%%

%%%%%%%%%%%%%%%%%%%%%%%%%%%%%%%%%%%
\subsection{Auxiliary estimates}
%%%%%%%%%%%%%%%%%%%%%%%%%%%%%%%%%%%
Let $\mathcal D$ be as in \eqref{b1a}. It is straightforward to see that 
$\mathcal D\subset \Dom(X_\gamma^{-1})$ for all $\gamma>0$. Denote $G_a=GE_0((0,a))$.
% Below we use systematically that the operators  $X_\gamma$ and $H_0$  commute.

\begin{lemma}\label{lma.d1}
Let Assumption~\ref{as1} hold true. 
Then for any $\gamma>0$, the operator 
$GR_0(i)X_\gamma^{-1}$, defined initially on $\mathcal D$, extends 
to a compact operator from $\calH$ to $\calK$. 
\end{lemma}

\begin{proof} 
By the definition \eqref{c8} of $X_\gamma$, we have
%\begin{equation}GR_0(i)X_\gamma^{-1}=GR_0(i)(\calF_a^* (X_\gamma^{(1)})^{-1}\calF_a\oplus I),\label{d11}\end{equation}
%where the orthogonal sum is taken with respect to the decomposition
%$\calH=\calH_a\oplus \calH_a^\bot$, $\calH_a=\Ran E_0((0,a))$. We need
 to prove the compactness of the two operators
\begin{equation}
G R_0(i) \calF_a^*(X_\gamma^{(1)})^{-1}\calF_a E_0((0,a))
\text{ and }
GR_0(i)E_0(\R\setminus(0,a)).
\label{d17}
\end{equation}
The second operator is compact by   assumption \eqref{a3}. Consider the first one.
Since the operators $H_0$ and 
$X_\gamma$ commute and $\calF_a$ is unitary, 
it suffices to prove the compactness of the operator
$G_a \calF_a^*(X_\gamma^{(1)})^{-1}: L^2((0,a),\calN)\to  \calK$.
%\label{d15}
According to formula \eqref{a3B}
this operator acts    as
\begin{equation}
G_a\calF_a^*(X_\gamma^{(1)})^{-1}f
=
\int_0^a (1+\abs{\log x})^\gamma Z(x)^*f(x)dx,
\quad
f\in\mathcal D.
\label{d16}
\end{equation}
By the strong smoothness assumption the operator 
$Z(x): \calK\to \calN$ is compact and depends  continuously on $x$.
From here and the fact that 
$(1+\abs{\log x})^\gamma$ 
is in $L^2((0,a),dx)$, the required statement follows. 
\end{proof}

Using the above lemma, 
we immediately obtain that
 for all $\lambda<0$ the operators
$GR_0(\lambda)E_0(\R_+)X_\gamma^{-1}$  
defined initially on 
$\mathcal D$ extend  to  compact operators from $\calH$ to $\calK$.

\begin{lemma}\label{lma.d1a}
Under Assumption~\ref{as1} for any $\gamma>0$ we have:
\begin{enumerate}[\rm (i)]

\item
$\norm{GR_0(\lambda)E_0((a,\infty))X_\gamma^{-1}}=O(1)$, 
as $\lambda\to-0$;

\item
$\norm{GR_0(\lambda)E_0(\R_+)X_\gamma^{-1}}
=O(\abs{\lambda}^{-1/2}\abs{\log\abs{\lambda}}^\gamma)$,
as $\lambda\to-0$.
\end{enumerate}
\end{lemma}

\begin{proof}
(i) 
%Recall that by the definition \eqref{c8} of $X_\gamma$, we have
%\begin{equation} E_0((a,\infty))X_\gamma^{-1}=E_0((a,\infty)).
%\label{d14}\end{equation} 
Since (in view of \eqref{omega})
$$
GR_0(\lambda)E_0((a,\infty))X_\gamma^{-1}
% =GR_0(\lambda)E_0((a,\infty))
=GR_0(i)(H_0-i)R_0(\lambda)E_0((a,\infty))
$$
and the operator $GR_0(i)$ is bounded,  the required
statement follows from the trivial estimate
$$
\norm{ (H_0-i)R_0(\lambda)E_0((a,\infty))}\leq C,
\quad \forall \lambda<0. 
$$

(ii)
It follows   from \eqref{c8} that  
% $$GR_0(\lambda)E_0(\R_+)X_\gamma^{-1}=
 % GR_0(\lambda)E_0(\R_+)
% (\calF_a^*(X_\gamma^{(1)})^{-1}\calF_a\oplus I)$$ and 
the problem reduces (cf. \eqref{d17}) to estimating the norms of the  two operators:
$$
\text{ $G_aR_0(\lambda)\calF_a^*(X_\gamma^{(1)})^{-1}$
and 
$GR_0(\lambda)E_0((a,\infty))$.}
$$
The norm of the second operator has already been estimated in (i).
Consider the first operator. According  to   \eqref{a3B}
this operator acts  from  $L^2((0,a),\calN)$  to $\calK$ as
$$
G_aR_0(\lambda)\calF_a^*(X_\gamma^{(1)})^{-1}f
=
\int_0^a
\frac{(1+\abs{\log x})^\gamma}{x-\lambda} Z(x)^* f(x)dx,
\quad
f\in\mathcal D, \quad \lambda<0.
$$
The norm of this operator can be explicitly estimated: 
\begin{multline*}
\Norm{ \int_0^a
\frac{(1+\abs{\log x})^\gamma}{x-\lambda} Z(x)^* f(x)dx}
\leq 
\norm{f}
\left(\int_0^a \frac{(1+\abs{\log x})^{2\gamma}}{(x-\lambda)^2}
\norm{Z(x)}^2dx\right)^{1/2}
\\
\leq C\norm{f}
\left(\int_0^a \frac{(1+\abs{\log x})^{2\gamma}}{(x-\lambda)^2}
dx\right)^{1/2}
\leq
C_1\norm{f}\abs{\lambda}^{-1/2}\bigl| \log\abs{\lambda}\bigr|^\gamma,
\end{multline*}
for all $\lambda<0$, and the required statement follows. 
\end{proof}

%Remark that the first statement of this lemma requires 
%only part (B) of Assumption~\ref{as1}.

%%%%%%%%%%%%%%%%%%%%%%%%%%%%%%%%%%%
\subsection{Compactness properties of  auxiliary operators}
%%%%%%%%%%%%%%%%%%%%%%%%%%%%%%%%%%%

%%%%%%%%%%%%%%%%%
\begin{lemma}\label{lma.d2}
%%%%%%%%%%%%%%%%%
Let Assumption~\ref{as1} hold true. 
Then for any $\varphi\in C_0^\infty(\R)$ 
and any $\gamma>0$  the operator 
$$
X_\gamma^{-1}(\varphi(H)-\varphi(H_0))
$$
is compact. 
\end{lemma}

\begin{proof}
1. First note that 
$$
X_\gamma^{-1}(GR_0(z))^*=
X_\gamma^{-1}(GR_0(i))^*
+
(\overline{z}+i)X_\gamma^{-1}(GR_0(i)R_0(z))^*,
$$
for $\Im z\not=0$. 
Using Lemma~\ref{lma.d1}, from here we get 
\begin{equation}
\norm{X_\gamma^{-1}(GR_0(z))^*}
\leq 
C
\frac{\abs{z}+1}{\abs{\Im z}}, 
\quad \Im z\not=0.
\label{d6}
\end{equation}
Next, from \eqref{a3} and \eqref{a9} it follows that $GR(i)$ is bounded. 
Therefore,
similarly to \eqref{d6}, we get
\begin{equation}
\norm{GR(z)}\leq C
\frac{\abs{z}+1}{\abs{\Im z}}, 
\quad \Im z\not=0.
\label{d7}
\end{equation}

2. 
We use the technique of functional calculus via the 
almost analytic extension, see e.g. \cite[Section~8]{DiSj}.
Let $\wt \varphi\in C_0^\infty(\C)$ be the almost analytic extension of $\varphi$, 
i.e. $\wt \varphi|_\R=\varphi$ and 
\begin{equation}
\Abs{\frac{\partial \wt \varphi}{\partial \overline z} (z)}
=
O(\abs{\Im z}^k)
\quad \text{as $\Im z\to0$}
\label{d1} 
\end{equation}
for any $k>0$. Then 
\begin{equation}
\varphi(H)
=
\int_{\C}\frac{\partial \wt \varphi}{\partial \overline z} (z) R(z) dL(z), 
\label{d5}
\end{equation}
where $L(z)$ is the Lebesgue measure in $\C$. 
Note that this integral is norm convergent due to \eqref{d1} and 
the trivial estimate $\norm{R(z)}\leq \abs{\Im z}^{-1}$.

3. 
Using the resolvent identity \eqref{res} and the representation \eqref{d5}, we get
$$
X_\gamma^{-1}(\varphi(H)-\varphi(H_0))
=
-
\int_{\C} 
\frac{\partial \wt \varphi}{\partial \overline z} (z)
X_\gamma^{-1} (GR_0(\overline z))^* JGR(z) dL(z).
$$
The integrand in the r.h.s. is compact for any $\Im z\not=0$  by Lemma~\ref{lma.d1}. 
By \eqref{d6}, \eqref{d7} and \eqref{d1}, the integral converges 
in the operator norm. 
From here we get the required statement. 
 \end{proof}

%%%%%%%%%%%%%%%%%
\begin{lemma}\label{lma.d4}
%%%%%%%%%%%%%%%%%
Let part (B) of Assumption~\ref{as1} hold true.  
Then the operator $\psi(H)-\psi(H_0)$ 
is compact for any function $\psi\in C(\R)$ 
such that the limits $\lim_{x\to\pm\infty}\psi(x)$
exist and are finite. 
\end{lemma}

\begin{proof}
As is well known (and can easily  be deduced from the compactness of $R(z)-R_{0}(z)$ for $\Im z\neq 0$),
the operator $\psi(H)-\psi(H_0)$ 
is compact for any function $\psi\in C(\R)$ 
such that $\psi(x)\to0$ as $\abs{x}\to\infty$. 
Therefore, it suffices to prove that $\psi(H)-\psi(H_0)$
is compact for at least one function $\psi\in C(\R)$ 
such that $\lim_{x\to\infty}\psi(x)\not=\lim_{x\to-\infty}\psi(x)$
and both limits exist. 
The latter fact is provided by  \cite[Theorem~7.3]{Push2} where 
it has been proven that if part (B) of Assumption~\ref{as1}
holds true then the difference $\tan^{-1}(H)-\tan^{-1}(H_0)$ is compact. 
\end{proof}

%%%%%%%%%%%%%%%%%%%%%%%%%%%%%%%%%%%
\subsection{Proofs of Lemmas~\ref{lma.b4}, \ref{lma.b5} and \ref{lma.b6}}
%%%%%%%%%%%%%%%%%%%%%%%%%%%%%%%%%%%

\begin{proof}[Proof of Lemma~\ref{lma.b4}]
1. Comparing \eqref{a5} and \eqref{M2}, we see that $M_+ - M_2 =X_\gamma K X_\gamma$, where 
$$
K=X_\gamma^{-1}E_0(\R_+) E((-\infty,-a))E_0(\R_+) X_\gamma^{-1}.
$$
It suffices to prove that 
the operator  
\begin{multline*}
X_\gamma^{-1}E_0(\R_+)E((-\infty,-a))
\\ =
X_\gamma^{-1}E_0((0,a))E((-\infty,-a))
+
X_\gamma^{-1}E_0((a,\infty))E((-\infty,-a))
\end{multline*}
is compact. 
We will prove the compactness of the two terms
in the r.h.s.  separately.

2. Consider the first term. Let $\varphi\in C_0^\infty(\R)$ be
such that $\varphi(x)=1$ for $x\in[0,a]$
and $\varphi(x)=0$ for $x\leq -a$. 
Then 
\[
X_\gamma^{-1}E_0((0,a))E((-\infty,-a))
\\
=X_\gamma^{-1}E_0((0,a))(\varphi(H_0)-\varphi(H))E((-\infty,-a)).
\]
Since the operators $E_0((0,a))$ and $X_\gamma^{-1}$ commute, the r.h.s. is compact by Lemma~\ref{lma.d2}. 

3. Consider the second term. 
Let $\psi\in C(\R)$ be such that 
$\psi(x)=1$ for $x\geq a$ and $\psi(x)=0$
for $x\leq -a$. Then,   using \eqref{omega}, we find that
$$
X_\gamma^{-1} E_0((a,\infty))E((-\infty,-a))
=
E_0((a,\infty))(\psi(H_0)-\psi(H))E((-\infty,-a)),
$$
and the r.h.s. is compact by Lemma~\ref{lma.d4}.
\end{proof}

\begin{proof}[Proof of Lemma~\ref{lma.b5}]

1. 
First we need to obtain an integral representation for $M_2$ 
similar to   \eqref{b1b}. 
By using  Stone's formula
(see e.g. \cite[Theorem~VII.13]{RS1}) 
and the fact 
that the spectra of $H_0$ and $H$ on $[-a,a]$ are 
purely a.c.,
we obtain
for any $f\in\calH$:
\begin{multline*}
(M_2f,f)
=
\bigl((E((-a,0))-E_0((-a,0)))E_0(\R_+)f,E_0(\R_+)f\bigr)
\\
=\frac1\pi  \int_{-a}^0 \lim_{\varepsilon\to+0} 
\Im  ((R(\lambda+i\varepsilon)-R_0(\lambda+i\varepsilon))E_0(\R_+)f,E_0(\R_+)f)
d\lambda.
\end{multline*}
Substituting the resolvent identity \eqref{a7} into this formula and 
using the notation $Y(\lambda)=J(I+T(\lambda+i0)J)^{-1}$, we obtain:
\begin{equation}
(M_2f,f)
=
-\frac1\pi \int_{-a}^0
\bigl((\Im Y(\lambda)) GR_0(\lambda)E_0(\R_+)f,GR_0(\lambda)E_0(\R_+)f\bigr)
d\lambda.
\label{d12}
\end{equation}

2. 
Comparing \eqref{M2} and \eqref{M3} and taking into account \eqref{omega}, 
we find that $ M_2-M_3 =X_\gamma K X_\gamma$, where 
\begin{equation}
K
=
 E_0((a,\infty)) M_2X_\gamma^{-1}
+
X_\gamma^{-1}M_2 E_0((a,\infty)) 
+
 E_0((a,\infty)) M_2 E_0((a,\infty)) .
\label{d8}
\end{equation}
Since $X_\gamma$ is a bounded operator, 
it suffices to check the compactness of the first    operator in the r.h.s.
By \eqref{d12}, it can be represented as 
\begin{equation}
-\frac1\pi \int_{-a}^0 
(G R_0(\lambda) E_0((a,\infty)))^*
\Im Y(\lambda) 
G R_0(\lambda) E_0(\R_+) X_\gamma^{-1} d\lambda,
\label{D8}\end{equation}
where \emph{a priori} the integral converges weakly on the dense set $\mathcal{D}$. 
Applying  Lemma~\ref{lma.d1a}, we see that the norm of integrand in \eqref{D8} is bounded by
\[
\| G R_0(\lambda) E_0((a,\infty)) \|\,
\|\Im Y(\lambda) \| \,
\|G R_0(\lambda) E_0(\R_+) X_\gamma^{-1}\|\leq C \abs{\lambda}^{-1/2}\abs{\log\abs{\lambda}}^\gamma. 
\]
Hence the integral in \eqref{D8} converges
 actually  in the operator norm. 
By Lemma~\ref{lma.d1}, the integrand is compact 
for all $\lambda<0$. 
Thus, the above integral is compact, as required. 
\end{proof}

% Now the compactness of the second operator in the r.h.s. of 
%\eqref{d8} follows by conjugation. 
%The compactness of the last term in the r.h.s. of \eqref{d8}
%follows from the compactness of the first term, since
%$X_\gamma^{-1}$ and $E_0((a,\infty))$ commute.

\begin{proof}[Proof of Lemma~\ref{lma.b6}]
Similarly to  \eqref{d12}, 
we have the representation
$$
(M_3f,f)=
-\frac1\pi \int_{-a}^0
\bigl((\Im Y(\lambda)) G R_0(\lambda) E_0((0,a))f, 
G R_0(\lambda) E_0((0,a))f\bigr) \, d\lambda.
$$
Thus, recalling the definition \eqref{b1b} of $M_4$
and setting $\widetilde Y(\lambda)=\Im (Y(\lambda)-Y(0))$, we get
$M_3-M_4=X_\gamma K X_\gamma$, where
\begin{multline}
(Kf,f)
\\
=-\frac1\pi \int_{-a}^0 
%(\Im (Y(\lambda)-Y(0)) 
(\widetilde Y(\lambda)
G R_0(\lambda) E_0((0,a))X_\gamma^{-1}f,  G R_0(\lambda) E_0((0,a))X_\gamma^{-1}f)
\, d\lambda,\quad f\in \mathcal D.
\label{d3}
\end{multline}
Since $Y(\lambda)$ is H\"older 
continuous, we have
$
\norm{\widetilde Y(\lambda) }\leq C\abs{\lambda}^\beta$, 
$\beta>0$.
Combining this with the estimate of Lemma~\ref{lma.d1a}(ii), we 
see that 
$$
\int_{-a}^0 
\norm{\widetilde Y(\lambda)}\norm{G R_0(\lambda) E_0((0,a))X_\gamma^{-1}}^2
\, d\lambda
<\infty.
$$
Recalling Lemma~\ref{lma.d1}, 
we obtain that the operator $K$ is compact. 
This result also shows that the operator $M_4$ is bounded. 
\end{proof}

%%%%%%%%%%%%%%%%%%%%%%%%%%%%%%%%%%%
\subsection{Proof of Lemma~\ref{lma.b7}}
%%%%%%%%%%%%%%%%%%%%%%%%%%%%%%%%%%%
First we need the following simple auxiliary statement. 
% Let, as in Section~\ref{sec.c},  $\calC_a(x,y)=1/(\pi(x+y))$.
%%%%%%%%%%%%%%%%%
\begin{lemma}\label{lma.d6}
%%%%%%%%%%%%%%%%%
Let $p>q>0$. Then the operator $K$ in $L^2(0,a)$
with the integral kernel
$$
K(x,y)
=
(1+\abs{\log x})^{-p} (x +y)^{-1}  (1+\abs{\log y})^{q}
$$
is compact.
\end{lemma}
The proof is given in the Appendix.

\begin{proof}[Proof of Lemma~\ref{lma.b7}]
1. First recall the definitions \eqref{c8} of $M $ and $X_\gamma$ and \eqref{b1b} of $M_{4}$. 
Next, note that both $M $ and $M_4$ vanish on $\calH_a^\bot$. 
Thus, applying a  unitary transformation $\calF_a$,  it suffices 
to prove that the operator 
$M_1-\calF_a M_4\calF_a^*$ in $L^2((0,a),\calN)$ 
can be represented as $X_\gamma^{(1)} K X_\gamma^{(1)}$
with a compact operator $K$. 

2. Consider  $M_1$ and $\calF_a M_4\calF_a^*$ as integral operators 
in $L^2((0,a),\calN)$. Set $Q=-\pi \Im Y(0)$. It follows from  the representations
\eqref{a6}, \eqref{uzuz1} that
$$
\Gamma=
-\pi Z(0)\Im Y(0) Z(0)^*
=
Z(0)QZ(0)^*.
$$ 
Therefore formula \eqref{uzu} shows that
the  integral kernel of $M_1$ can be represented 
as 
$$
M_1 (x,y)=\calC_a^2(x,y) Z(0)Q Z(0)^*
$$
where $\calC_a^2(x,y)$ is defined by \eqref{uzu}.
Next, it follows from \eqref{a3B} that
$$
G R_0(\lambda)E_0((0,a))\calF_a^*  f
=
\int_0^a \frac{1}{x-\lambda}Z(x)^* f(x)dx, 
\quad \lambda<0.
$$
From here and the definition \eqref{b1b} of $M_4$ it is clear that 
the integral kernel of 
$\calF_a M_4 \calF_a^*$ is
\begin{align*}
(\calF_a M_4 \calF_a^*)(x,y)
&=
\frac1{\pi^2} \int_{-a}^0 Z(x) \frac1{x-\lambda}Q \frac1{y-\lambda} Z(y)^*\, d\lambda
\\
&=
\calC_a^2(x,y) Z(x) Q Z(y)^*.
\end{align*}
Using the definition \eqref{c12}, \eqref{uzuz} of $X_\gamma^{(1)}$,
let us represent the integral kernel of the difference
$$
(X_\gamma^{(1)})^{-1}(M_1 -\calF_a M_4 \calF_a^*)(X_\gamma^{(1)})^{-1}
$$
as
\begin{multline*}
\calC_a^2(x,y)(\omega_\gamma(x)\omega_\gamma(y))^{-1}(Z(x)-Z(0)) Q Z(y)^*
\\
+
\calC_a^2(x,y)(\omega_\gamma(x)\omega_\gamma(y))^{-1} Z(0) Q (Z(y)^*-Z(0)^*)
\end{multline*}
where $\omega_\gamma(x)= (1+\abs{\log x})^{-\gamma}$.

3. Let us prove that the first kernel represents a compact operator; 
the second kernel is considered in the same way. 
 We have
\begin{multline}
\calC_a^2(x,y)(\omega_\gamma(x)\omega_\gamma(y))^{-1}(Z(x)-Z(0)) Q Z(y)^*
\\
=
\frac1{\pi^2}\int_0^a \omega_\gamma(x)^{-1}(Z(x)-Z(0))\frac1{x+t} Q\frac1{t+y} Z(y)^* \omega_\gamma(y)^{-1} dt.
\label{d4}
\end{multline}
Choose $\sigma>\gamma$. 
The above  formula defines a factorization of the operator with integral kernel \eqref{d4}
as $K_1K_2$, where 
\begin{align*}
K_1:L^2((0,a),\calK)& \to L^2((0,a),\calN),
\\
K_2:L^2((0,a),\calN)& \to L^2((0,a),\calK)
\end{align*}
are the integral operators with the kernels
\begin{align*}
K_1(x,t)&=\frac1{\pi } \omega_\gamma(x)^{-1}(Z(x)-Z(0)) \frac1{x+t} \omega_\sigma(t)^{-1} Q,
\\
K_2(t,y)&= \frac1{\pi } \omega_\sigma(t) \frac1{t+y} Z(y)^* \omega_\gamma(y)^{-1}.
\end{align*}
Using Lemma~\ref{lma.d6} and the fact that $Z(y)^*$ is compact for all $y$, 
we see that the operator $K_2$ is compact. 
  Since $\norm{Z(x)-Z(0)}<C\abs{x}^\alpha$,   similar arguments yield
  the compactness of $K_1$. 
\end{proof}

%%%%%%%%%%%%%%%%%%%%%%%%%%%%%%%%%%%%%%%%%%%%%%%%%
%%%%%%%%%%%%%%%%%%%%%%%%%%%%%%%%%%%%%%%%%%%%%%%%%
\appendix
\section{}
%%%%%%%%%%%%%%%%%%%%%%%%%%%%%%%%%%%%%%%%%%%%%%%%%
%%%%%%%%%%%%%%%%%%%%%%%%%%%%%%%%%%%%%%%%%%%%%%%%%
%\renewcommand{\theequation}{A.\arabic{equation}}
%\renewcommand{\thetheorem}{A.\arabic{theorem}}
%\setcounter{theorem}{0}
%\setcounter{equation}{0}

Here we prove three elementary statements.

\begin{lemma}\label{lma.app}
Let Assumption~\ref{as1} hold true. 
Then the spectrum of $H$ on $\Omega$
is purely a.c.
\end{lemma}
\begin{proof}
Let  $\Delta_{n } = (a_{n}, b_{n})$ be one of the component intervals of the open set $\Omega$. 
It suffices to prove that for every $\varepsilon>0$ and  a dense set of elements $f\in\calH$, 
the function $(R(z)f,f)$ is  bounded on the set  
$$
\Pi_{n,\varepsilon}:=\{z\in\C\mid \Re z\in[ a_{n}+\varepsilon, b_{n}-\varepsilon],\  \Im z\in(0,1)\}.
$$
Using \eqref{a7}, write 
\begin{equation}
(R(z)f,f)=(R_0(z)f,f)
+
(J(I+T(z)J)^{-1} GR_0(z)f,GR_0(z)f).
\label{A1}
\end{equation}
Evidently,  the norms of $(I+T(z)J)^{-1}$ are uniformly 
bounded 
for all $z\in\Pi_{n,\varepsilon}$. 
For $f\in\Ran E_0(\R\setminus\Delta_{n})$, 
it is obvious that $R_0(z)f$ and hence the r.h.s. of \eqref{A1} 
is   bounded
for $z\in\Pi_{n,\varepsilon}$. 
Next, let $L_{n}\subset \Ran E_0(\Delta_{n})$
be the set of elements $f$ such that 
$\calF f\in C_{0}^\infty (\Delta_{n},\mathcal{N})$ 
(recall that $\calF$ is defined in \eqref{a8}).  
It is clear that $L_{n}$ is dense  in $\Ran E_0(\Delta_{n})$. 
It follows from \eqref{a8}, \eqref{a3b} that for all $f\in L$ and $g\in \mathcal{H}$
\begin{equation}
(R_0(z)f,f)
=
\int_{{\Delta}_{n}}
\frac{\| (\calF f)(\lambda)\|^2_{\mathcal{N}}}{\lambda-z}d \lambda,
\label{A1a}
\end{equation}
and
\begin{equation}
(GR_0(z)f,g)
=
\int_{{\Delta}_{n}}
\frac{((\calF f)(\lambda), Z(\lambda) g)_{\mathcal{N}} }{\lambda-z}d \lambda.
\label{A1b}
\end{equation}
According to \eqref{a3a},  $((\calF f)(\lambda), Z(\lambda) g)_{\mathcal{N}} $ is a
H\"older continuous function of $\lambda\in {\Delta}_{n}$; 
moreover, the corresponding constant in the definition of H\"older continuity 
is bounded by $C\|g\|$. 
 Therefore, by the Privalov theorem, integral \eqref{A1b} is bounded by $C\|g\|$   
for all $z\in\Pi_{n,\varepsilon}$. Hence the function $\|GR_0(z)f\|$ is     bounded
on $\Pi_{n,\varepsilon}$. Integral \eqref{A1a} is considered in a similar but simpler way.
Thus,  the r.h.s. of \eqref{A1} 
is   bounded on $\Pi_{n,\varepsilon}$.  This proves the required statement.
\end{proof}

\begin{proof}[Proof of Lemma~\ref{lma2}]
We recall that the Legendre function can be expressed in terms of the 
hypergeometric function as 
\begin{equation}
P_\nu(x)=F(-\nu,\nu+1;1;\tfrac{1-x}{2}), \quad \abs{x-1}<2.
\label{c2}
\end{equation}
The hypergeometric function $F(a,b;c; z)$   is defined by the hypergeometric series
\begin{equation}
F(a,b;c;z)=
\sum_{n=0}^\infty \frac{(a)_n (b)_n}{(c)_n} \frac{z^n}{n!},
\quad (a)_n=a(a+1)\cdots(a+n-1).
\label{c9}
\end{equation}
For $\abs{z}<1$, this series is absolutely convergent and
analytic in $a$, $b$, $c$, $z$. 
%The analytic continuation of the hypergeometric function outside the unit disc $\abs{z}<1$ is given by various  transformation formulas. Applying one of such formulas 
For $x>1$, formulas (9) and (23) of Section 3.2 of  \cite{BE}  yield the representation
\begin{multline}
P_{-\frac12+it}(x)
=
\frac{\Gamma(-it)}{\sqrt{\pi}\Gamma(\frac12-it)}(2x)^{-\frac12-it}
F(\tfrac14+i\tfrac{t}{2},\tfrac34+i\tfrac{t}{2}; 1+it; x^{-2})
\\
+\frac{\Gamma(it)}{\sqrt{\pi}\Gamma(\frac12+it)}(2x)^{-\frac12+it}
F(\tfrac14-i\tfrac{t}{2},\tfrac34-i\tfrac{t}{2}; 1-it; x^{-2}).
\label{c10}
\end{multline}
Let us split the interval $[1,\infty)$ into $[1,2)$ and $[2,\infty)$. 
For $x\in[1,2)$, we can use \eqref{c2};  then 
$\abs{\frac{1-x}2}<1/2$ and so the hypergeometric
series converges uniformly which shows that 
the estimates \eqref{c6}, \eqref{c7} are trivially true in this range
of $x$. 

For $x\in[2,\infty)$, we can use \eqref{c10} and expand the hypergeometric
function in the r.h.s. in the hypergeometric series. 
The series converges uniformly in $x\in[2,\infty)$. 
Observing that $F(a,b;c;0)=1$ and using  the elementary estimate
$$
\abs{x^{it_1}-x^{it_2}}\leq C(\delta) \abs{t_2-t_1}^\delta (\log x)^\delta, 
$$
we obtain the estimates \eqref{c6}, \eqref{c7} for $x\geq 2$. 
\end{proof}

\begin{proof}[Proof of Lemma~\ref{lma.d6}]
1. For $\delta\in(0,a)$, let $\chi_\delta$ be the characteristic
function of the interval $(0,\delta)$ and let 
$\wt \chi_\delta=1-\chi_\delta$. 
Along with $K$, consider the integral operator $\wt K_\delta$
with the integral kernel $\wt K_\delta(x,y)=K(x,y)\wt \chi_\delta(y)$. 
A direct inspection shows that the kernel $\wt K_\delta(x,y)$ is 
uniformly bounded in $(x,y)\in[0,a]\times[0,a]$ and therefore
the operator $\wt K_\delta$ is in the Hilbert-Schmidt class. 
Thus, it suffices to show that 
\begin{equation}
\norm{K-\wt K_\delta}\to0
\quad \text{as $\delta\to0$.}
\label{d9}
\end{equation}

2. 
Let $K_\delta=K-\wt K_\delta$ and $f,g\in L^2(0,a)$. 
Using Cauchy-Schwartz, we have
\begin{multline}
\abs{(K_\delta f,g)}
\leq
\frac1\pi
\int_0^a dx \int_0^\delta dy
\frac1{x+y}(1+\abs{\log x})^{-p}(1+\abs{\log y})^{q}
\sqrt[4]{\frac{x}{y}}\sqrt[4]{\frac{y}{x}}
\abs{f(y)}\abs{g(x)}
\\
\leq
\frac1\pi
\left(\int_0^a dx \int_0^\delta dy \frac1{x+y} \sqrt{\frac{y}{x}}
\abs{f(y)}^2 \right)^{1/2}
\\
\times
\left(\int_0^a dx \int_0^\delta dy \frac1{x+y} \sqrt{\frac{x}{y}}
\frac{(1+\abs{\log y})^{2q}}{(1+\abs{\log x})^{2p}} \abs{g(x)}^2\right)^{1/2}.
\label{d10}
\end{multline}
Next, we have
$$
\sup_{0<y<a} \int_0^a \frac1{x+y}\sqrt{\frac{y}{x}}dx
\leq
\int_0^\infty \frac1{x+1} \frac1{\sqrt{x}}dx=C<\infty, 
$$
and therefore   the first term in the r.h.s. of \eqref{d10} is  bounded by
$ C\norm{f}$.
In order to estimate the second term in the r.h.s. of \eqref{d10}, 
we first note the elementary estimate
$$
(1+\abs{\log(xy)})
\leq
(1+\abs{\log x})(1+\abs{\log y}).
$$
Using this, we have:
\begin{multline*}
\sup_{0<x<\delta}
\int_0^\delta \frac1{x+y} \sqrt{\frac{x}{y}}
\frac{(1+\abs{\log y})^{2q}}{(1+\abs{\log x})^{2p}} dy
\\
\leq
(1+\abs{\log \delta})^{2q-2p}
\sup_{0<x<\delta}
\int_0^\delta \frac1{x+y} \sqrt{\frac{x}{y}}
\frac{(1+\abs{\log y})^{2q}}{(1+\abs{\log x})^{2q}} dy
\\
=
(1+\abs{\log \delta})^{2q-2p}
\sup_{0<x<\delta}
\int_0^{\delta/x}\frac1{1+t}\frac1{\sqrt{t}}
\frac{(1+\abs{\log ( xt)})^{2q}}{(1+\abs{\log x})^{2q}}dt
\\
\leq
(1+\abs{\log\delta})^{2q-2p}
\int_0^\infty
\frac1{1+t}\frac1{\sqrt{t}}
(1+\abs{\log t})^{2q} dt
=
C(1+\abs{\log\delta})^{2q-2p}.
\end{multline*}
From here we get the estimate for the second term 
in the r.h.s. of \eqref{d10} by $C(1+\abs{\log\delta})^{q-p} \norm{g}$.
Thus, we have
$$
\abs{(K_\delta f,g)}\leq C(1+\abs{\log\delta})^{q-p} \norm{f}\norm{g},
$$
and \eqref{d9} follows.
\end{proof}

\section*{Acknowledgements}
Our collaboration has become possible through the hospitality and
financial support of the Departments of Mathematics of the University of Rennes 1 
and  of King's College London.
The financial support by the European Group of Research SPECT 
is also gratefully acknowledged.

\end{document}